\documentclass{article}
\usepackage{amsmath,amsthm,amssymb,geometry,enumerate,indentfirst}
\title{Sum of Two Squares in Biquadratic Fields}
\author{Wenhuan Huang}
\theoremstyle{definition}
\newtheorem{theorem}{Theorem}[section]

\newtheorem{lemma}[theorem]{Lemma}
\newtheorem{example}[theorem]{Example}
\newtheorem*{remark}{Remark}

\begin{document}

\maketitle

\begin{abstract}
This paper gives an algorithm to determine whether a number in a biquadratic field is a sum of two squares,
 based on local-global principle of isotropy of quadratic forms.

\end{abstract}

\section*{Introduction}

Let $F_1=\mathbb{Q}(\sqrt{a})$, $F_2=\mathbb{Q}(\sqrt{b})$ be two quadratic fields, with
$a$, $b$ square-free integers and $ab$ non-square. Put $c=\frac{ab}{gcd(a,b)^2}$ and 
$F=F_1F_2$, $F_3=\mathbb{Q}(\sqrt{c})$, then $c$ is also square-free, and $F_1$, $F_2$ and $F_3$
are the three quadratic subfields of $F$. Let $\sigma_i$ be the elements of $Gal(F/\mathbb{Q})\simeq(\mathbb{Z}/2)^2$ 
other than $id$ satisfying $\sigma_i|_{F_i}=id|_{F_i}$.

Some critical values about $F$, including integers in $F$ and discriminant of $F$, are already given in [1].

In this article, using local-global principle of isotropy of quadratic forms (see 66:1 of [2]),
we will give an explicit algorithm to determine whether a number $s_0+s_1\sqrt{a}+s_2\sqrt{b}+s_3\sqrt{c}$
($s_0,s_1,s_2,s_3\in\mathbb{Z}$) is a sum of two squares in ${F}$.
First we prepare some lemmas for convenience (see [4]):
\begin{lemma}
    Let $\mathbb{Q}_2(\sqrt{k})$ be a quadratic extension on $\mathbb{Q}_2$, $k\in\{2,3,7,10,11,14,5\}$, then $t\in \mathbb{Q}_2(\sqrt{k})$ is 
    a square, if and only if

    ($k=2$)$t=2^mr$, $m\geq0$ an integer, $\sqrt{2}\nmid r$, $r\equiv 1, 3+2\sqrt{2}\pmod{4\sqrt{2}}$.

    ($k=3$)$t=(\sqrt{3}-1)^{2m}r$, $m\geq0$ an integer, $(\sqrt{3}-1)\nmid r$, $r\equiv 1, 3\pmod {4(\sqrt{3}-1)}$.

    ($k=7$)$t=(3-\sqrt{7})^{2m}r$, $m\geq0$ an integer, $(3-\sqrt{7})\nmid r$, $r\equiv \pm1\pmod {4(3-\sqrt{7})}$.

    ($k=11$)$t=(\sqrt{11}-3)^{2m}r$, $m\geq0$ an integer, $(\sqrt{11}-3)\nmid r$, $r\equiv 1,3\pmod {4(\sqrt{11}-3)}$.

    ($k=14$)$t=(4-\sqrt{14})^{2m}r$, $m\geq0$ an integer, $(4-\sqrt{14})\nmid r$, $r\equiv \pm1\pmod {4(4-\sqrt{14})}$.

    ($k=10$)$t=Mr$, $v_{(2,\sqrt{10})}(M)$ is even, $v_{(2,\sqrt{10})}(r)=0$, $r\equiv 1, 3+2\sqrt{10}\pmod {4(2,\sqrt{10})}$.

    ($k=5$)$t=4^{m}r$, $m\geq0$ an integer, $2\nmid r$, $r\equiv 1,\frac{3\pm\sqrt{5}}{2}\pmod 4$.
\end{lemma}

\begin{lemma}
    Take assumptions as the last lemma and let $h$ be a unit in $O_{\mathbb{Q}_2(\sqrt{k})}$, then $h$ is the sum of two squares if and only if 
    
    ($k\neq 5$)$2|h-1$.

    ($k=5$)$h\equiv1,3,\frac{\pm3\pm\sqrt{5}}{2}\pmod 4$.
\end{lemma}

\section{Prime Numbers}

First we state some facts. For any odd prime number $p$, either $p\nmid abc$, or $p$ exactly divides two of $\{a,b,c\}$.
In the former case $(\frac{a}{p})(\frac{b}{p})=(\frac{c}{p})$ ($\frac{*}{*}$ denotes Legendre symbol),
and in the latter case, say $p|a$ and $p|b$, $(\frac{a/p}{p})(\frac{b/p}{p})=(\frac{c}{p})$.

It is well-known that a prime number $q$ is a sum of two squares in $\mathbb{Q}$ if and only if it $q\not\equiv3\pmod4$.
Let $q$ be a prime number with $q\equiv3\pmod4$. Next we find all biquadratic fields where $q$ becomes a sum of two squares.
\begin{lemma}
    (1)-1 is a sum of two squares in $F$ if $F$ is imaginary, unless $a$ and $b$ are both $\equiv1\pmod8$.

    (2)$q$ is a sum of two squares in $F$ unless $(\frac{a}{q})=(\frac{b}{q})=1$ or $a, b\equiv1\pmod8$.
\end{lemma}
\begin{proof}
    (1)We only need $(\frac{-1,-1}{\mathcal{P}})=1$ for $\mathcal{P}$ infinite or dyadic. By lemma 0.2, $(\frac{-1,-1}{\mathcal{P}})=1$
    if and only if $F_\mathcal{P}\not\simeq\mathbb{Q}_2$, i.e, $2O_F$ does not completely split.

    (2)Of course $(\frac{-1,q}{\mathcal{P}})=1$ at infinite places. We only need to discuss $q$-adic and dyadic cases. 
    By lemma 0.2, $(\frac{-1,q}{\mathcal{P}})=1$($\mathcal{P}$ dyadic) if and only if $2O_F$ does not completely split.
    In this case, if $(\frac{-1,q}{\mathcal{Q}})=-1$ for $q$-adic place $\mathcal{Q}$, then $F_\mathcal{Q}$ must 
    not contain $\mathbb{Q}_q(\sqrt{-1})$, $\mathbb{Q}_q(\sqrt{q})$ or $\mathbb{Q}_q(\sqrt{-q})$ (in the last case $(\frac{-1,-1}{\mathcal{Q}})=1$
    , making $(\frac{-1,q}{\mathcal{Q}})=(\frac{-1,-1}{\mathcal{Q}})(\frac{-1,-q}{\mathcal{Q}})=1$).
    So $F_\mathcal{Q}\simeq \mathbb{Q}_q$, i.e, $(\frac{a}{q})=(\frac{b}{q})=1$.
\end{proof}

\section{General cases}

Let $S=s_0+s_1\sqrt{a}+s_2\sqrt{b}+s_3\sqrt{c}\in F$ in introduction. By 66:1 and 63:12 of [2],
to decide whether $S$ is a sum of squares in $F$,
we only need to calculate $(\frac{-1,S}{\mathcal{P}})$s for all infinite, dyadic or $p$-adic places
for $p\equiv3\pmod4$ and $p|N_{F/\mathbb{Q}}(S)$. The infinite case is obvious: 
$(\frac{-1,S}{\mathcal{\infty}})=-1$ if and only if $F$ is real and $S$ is not totally positive.

Let $\mathcal{P}\cap\mathbb{Z}=p\equiv3\mod4$, $p|N_{F/\mathbb{Q}}(S)$. 

\subsection{Algorithm about unramified $p$}

First we talk about the case $p\nmid abc$.
If $(\frac{a}{p})=(\frac{b}{p})=(\frac{c}{p})=1$, then $pO_F$ completely splits, and 
four $(F_{\mathcal{P}})$s are isomorphic to $\mathbb{Q}_p$. set
$$T_i=S\sigma_i(S)(i=1,2,3)$$
Then $T_i\in F_i$, and $S$ is a sum of two squares implies that so are all $T_i$s in $F_i$s.
Conversely, if all $T_i$s are sum of two squares, then $(\frac{-1,T_i}{\mathcal{P}})$s
($\mathcal{P}$ finite and non-dyadic) are all 1, which implies that 
for every $p$-adic places with $p\equiv3\pmod4$, $F_{\mathcal{P}}$ contains a subfield isomorphic to $\mathbb{Q}_p(\sqrt{-1})$,
or $v_\mathcal{P}(T_i)$s are all even.

We need to state some facts about $p$-adic ideals. Let $A,B,C$ be (mod $p$) square roots of $a,b,c$ respectively,
satisfying $C\mathrm{gcd}(a,b)\equiv AB\pmod p$, then $\sqrt{c}-C\in (p,\sqrt{a}-A,\sqrt{b}-B)$.
Take 
$$\mathcal{P}_1=(p,\sqrt{a}-A,\sqrt{b}-B,\sqrt{c}-C)$$
$$\mathcal{P}_2=(p,\sqrt{a}-A,\sqrt{b}+B,\sqrt{c}+C)$$
$$\mathcal{P}_3=(p,\sqrt{a}+A,\sqrt{b}-B,\sqrt{c}+C)$$
$$\mathcal{P}_4=(p,\sqrt{a}+A,\sqrt{b}+B,\sqrt{c}-C)$$
Then 
$$pO_{F_1}=\mathfrak{p}_{11}\mathfrak{p}_{12},\mathfrak{p}_{11}O_F=(p,\sqrt{a}-A)O_F=\mathcal{P}_1\mathcal{P}_2,\mathfrak{p}_{12}O_F=(p,\sqrt{a}+A)O_F=\mathcal{P}_3\mathcal{P}_4$$
$$pO_{F_2}=\mathfrak{p}_{21}\mathfrak{p}_{22},\mathfrak{p}_{21}O_F=(p,\sqrt{b}-B)O_F=\mathcal{P}_1\mathcal{P}_3,\mathfrak{p}_{22}O_F=(p,\sqrt{b}+B)O_F=\mathcal{P}_2\mathcal{P}_4$$
$$pO_{F_1}=\mathfrak{p}_{11}\mathfrak{p}_{12},\mathfrak{p}_{31}O_F=(p,\sqrt{c}-C)O_F=\mathcal{P}_1\mathcal{P}_4,\mathfrak{p}_{32}O_F=(p,\sqrt{c}+C)O_F=\mathcal{P}_2\mathcal{P}_3$$
Thus we have 
\begin{eqnarray}    
    2v_{\mathcal{P}_1}(S)&=&v_{\mathcal{P}_1}(S\sigma_1(S))+v_{\mathcal{P}_1}(S\sigma_2(S))-v_{\mathcal{P}_1}(\sigma_1(S)\sigma_2(S)) \nonumber    \\
    ~&=&v_{\mathfrak{p}_{11}}(S\sigma_1(S))+v_{\mathfrak{p}_{21}}(S\sigma_2(S))-v_{\mathfrak{p}_{31}}(\sigma_1(S)\sigma_2(S)) \nonumber
\end{eqnarray}

Similarly, we have
$$2v_{\mathcal{P}_2}(S)=v_{\mathfrak{p}_{11}}(S\sigma_1(S))+v_{\mathfrak{p}_{22}}(S\sigma_2(S))-v_{\mathfrak{p}_{32}}(\sigma_1(S)\sigma_2(S))$$
$$2v_{\mathcal{P}_3}(S)=v_{\mathfrak{p}_{12}}(S\sigma_1(S))+v_{\mathfrak{p}_{21}}(S\sigma_2(S))-v_{\mathfrak{p}_{32}}(\sigma_1(S)\sigma_2(S))$$
$$2v_{\mathcal{P}_4}(S)=v_{\mathfrak{p}_{12}}(S\sigma_1(S))+v_{\mathfrak{p}_{22}}(S\sigma_2(S))-v_{\mathfrak{p}_{31}}(\sigma_1(S)\sigma_2(S))$$

Therefore given an algorithm to calculate powers of a number in a prime ideal of a quadratic field, that in a biquadratic field can be obtained.

\begin{remark}
    Here is an algorithm to calculate the power of a number in an explicit ideal of a quadratic field:

    Take $\mathfrak{p}_{11}=(p,\sqrt{a}-A)$ as above, with $pO_{F_1}$ splits and $r=x+y\sqrt{a}$, $x,y\in\mathbb{Z}$. If $p\nmid\mathrm{gcd}(x,y)$, then replace
    $\sqrt{a}$ to $A$ in $r$ we get $\hat{r}=x+yA$. If $p|\hat{r}$ then $v_{\mathfrak{p}_{11}}(r)=v_{p}(x^2-ay^2)$, otherwise $v_{\mathfrak{p}_{11}}(r)=0$.
    If $p^f||\mathrm{gcd}(x,y)$, then $v_{\mathfrak{p}_{11}}(r)=f+v_{\mathfrak{p}_{11}}(\frac{r}{p^f})$.
\end{remark}

\begin{lemma}
    If $p\equiv3\pmod4$ is a prime dividing $N_{F/\mathbb{Q}}(S)$ with $(\frac{a}{p})=(\frac{b}{p})=(\frac{c}{p})=1$,
    then $(\frac{-1,S}{\mathcal{P}})=1$ for all four $p$-adic places $\mathcal{P}$, if and only if 
    $$v_{\mathfrak{p}_{11}}(S\sigma_1(S))+v_{\mathfrak{p}_{21}}(S\sigma_2(S))-v_{\mathfrak{p}_{31}}(\sigma_1(S)\sigma_2(S))$$
    $$v_{\mathfrak{p}_{11}}(S\sigma_1(S))+v_{\mathfrak{p}_{22}}(S\sigma_2(S))-v_{\mathfrak{p}_{32}}(\sigma_1(S)\sigma_2(S))$$
    $$v_{\mathfrak{p}_{12}}(S\sigma_1(S))+v_{\mathfrak{p}_{21}}(S\sigma_2(S))-v_{\mathfrak{p}_{32}}(\sigma_1(S)\sigma_2(S))$$
    $$v_{\mathfrak{p}_{12}}(S\sigma_1(S))+v_{\mathfrak{p}_{22}}(S\sigma_2(S))-v_{\mathfrak{p}_{31}}(\sigma_1(S)\sigma_2(S))$$
    are all divided by 4.
\end{lemma}

If $(\frac{a}{p}),(\frac{b}{p}),(\frac{c}{p})$ is an arrangement of -1,-1,1 with $p\equiv3\pmod4$ then
$F_{\mathcal{P}}$s are all isomorphic to $\mathbb{Q}_p(\sqrt{-1})$, hence all $(\frac{-1,S}{\mathcal{P}})=1$.

\subsection{Algorithm about ramified $p$}

First set $p|abc$, say, $p|a$ and $p|b$.

If $(\frac{a/p}{p})\neq(\frac{b/p}{p})$, then $(\frac{c}{p})=-1$, 
so $F_\mathcal{P}$ has a subfield isomorphic to $\mathbb{Q}_p(\sqrt{-1})$,
making -1 a square and thus $(\frac{-1,S}{\mathcal{P}})=1$.

Assume $(\frac{a/p}{p})=(\frac{b/p}{p})$. Then we have explicit expression of $p$-adic ideals:
$$pO_{F_1}=\mathfrak{p}_1^2,\mathfrak{p}_1O_F=\mathcal{P}_1\mathcal{P}_2$$
$$pO_{F_2}=\mathfrak{p}_2^2,\mathfrak{p}_2O_F=\mathcal{P}_1\mathcal{P}_2$$
$$pO_{F_3}=\mathfrak{p}_{31}\mathfrak{p}_{32},\mathfrak{p}_{3i}O_F=\mathcal{P}_i^2$$
with $i\in\{1,2\}$, and 
$$\mathfrak{p}_1=(p,\sqrt{a}),\mathfrak{p}_2=(p,\sqrt{b}),\mathfrak{p}_{3i}=(p,\sqrt{c}+(-1)^iC)$$
where $C^2\equiv c\pmod p$, and 
$$\mathcal{P}_i=(p,\sqrt{a},\sqrt{b},\sqrt{c}+(-1)^iC).$$
Thus in order to determine $v_{\mathcal{P}_i}(S)$, we denote
\begin{eqnarray}   
    2v_{\mathcal{P}_1}(S)&=&v_{\mathcal{P}_1}(S\sigma_1(S))+v_{\mathcal{P}_1}(S\sigma_2(S))-v_{\mathcal{P}_1}(\sigma_1(S)\sigma_2(S)) \nonumber    \\
    ~&=&v_{\mathfrak{p}_{1}}(S\sigma_1(S))+v_{\mathfrak{p}_{2}}(S\sigma_2(S))-2v_{\mathfrak{p}_{31}}(\sigma_1(S)\sigma_2(S)) \nonumber
\end{eqnarray}

Similarly we have
$$2v_{\mathcal{P}_2}(S)=v_{\mathfrak{p}_{1}}(S\sigma_1(S))+v_{\mathfrak{p}_{2}}(S\sigma_2(S))-2v_{\mathfrak{p}_{32}}(\sigma_1(S)\sigma_2(S)).$$
\begin{remark}
    One can easily prove that if $x,y\in\mathbb{Z}$, $p^f||\mathrm{gcd}(x,y)$, then
    $$v_{\mathfrak{p}_1}(x+y\sqrt{a})=f+v_{p}(\frac{x^2-ay^2}{p^{2f}})=v_{p}(x^2-ay^2)-f.$$
\end{remark}

Thus we obtain that

\begin{lemma}
    If $p\equiv3\pmod4$ is a prime dividing $N_{F/\mathbb{Q}}(S)$, and $abc$ (say, $p|a$ and $p|b$), 
    and $(\frac{a/p}{p})=(\frac{b/p}{p})$ (equivalent that $(\frac{c}{p})=1$),
    then $(\frac{-1,S}{\mathcal{P}})=1$ for both $p$-adic places $\mathcal{P}$, if and only if
    both $v_{\mathfrak{p}_{1}}(S\sigma_1(S))+v_{\mathfrak{p}_{2}}(S\sigma_2(S))-2v_{\mathfrak{p}_{3i}}(\sigma_1(S)\sigma_2(S))$
    are divided by 4 for $i\in\{1,2\}$.
\end{lemma}

\subsection{Dyadic cases}

Finally we need to compute $(\frac{-1,S}{\mathcal{P}})$ for dyadic places $\mathcal{P}$. 
(Of course we assume $(\frac{-1,S}{\mathcal{P}})=1$ all infinite and non-dyadic places $\mathcal{P}$.)
If none of $a,b,c$ is $\equiv1\pmod8$ (For example, $(a,b,c)=(2,3,6)$ or $(2,5,10)$), then $F$ has only one dyadic place,
and by Hilbert Reciprocity Law we have $(\frac{-1,S}{\mathcal{P}})=1$.

Next we assume one of $\{a,b,c\}$, say $a$, is $\equiv1\pmod8$. Then $b\equiv c\pmod8$.

(1)If $b\not\equiv1\pmod4$, i.e, $2O_F$ splits and ramifies, then $F_\mathcal{P}$s are isomorphic to 
a quadratic ramified extension of $\mathbb{Q}_2$.
By lemma 0.2, $(\frac{-1,S}{\mathcal{P}})=1$ for both dyadic places $\mathcal{P}$, if and only if $S=2^f(2S'+1)$, with $f=v_2(S)$ and $S'\in O_F$. 
(Recall that algebraic integers in biquadratic fields are given in Theorem 2 of [1], e.g, $\frac{1}{4}(1-\sqrt{21}+\sqrt{33}+\sqrt{77})$
is an algebraic integer in $\mathbb{Q}(\sqrt{21},\sqrt{33})$.)

(2)If $b\equiv5\pmod8$, i.e, $2O_F$ splits and inerts, then both $F_\mathcal{P}$s are isomorphic to 
$\mathbb{Q}_2(\sqrt{5})$. Let $S''=\frac{S}{2^{v_2(S)}}$($0\leq v_2(S)\leq 2$) and
$$
    e(N)=\left\{
    \begin{array}{lll}
        E &   & if\ N\equiv E^2\pmod{256}\ odd\ and\ square-free, 1\leq E(W)\leq63\\ 
        \sqrt{5} &   & if\ N\equiv5\pmod{32}\\ 
        \sqrt{5}(1+2^2+2^3) &   & if\ N\equiv13\pmod{32}\\ 
        \sqrt{5}(1+2^3) &   & if\ N\equiv21\pmod{32}\\ 
        \sqrt{5}(1+2^2)  &   & if\ N\equiv29\pmod{32}\\ 
        s_1s_2\dots s_ge(W) &  & if\ w=s_1^2s_2^2\dots s_g^2W,W\equiv1\pmod4\ square-free
    \end{array}
\right.
    $$
(Recall Lemma 2.7 of [4].) Replace $\sqrt{b}$ with $e(b)$, $\sqrt{c}$ with $e(c)$, and $\sqrt{a}$ with $A$ that
$A\mathrm{gcd}(b,c)\equiv e(b)e(c)\pmod{16}$, then in $F_{\mathcal{P}}\simeq\mathbb{Q}_2(\sqrt{5})$ we have
$$S''\equiv \frac{1}{2^{v_2(S)}}(s_0+s_1A+s_2e(b)+s_3e(c))\pmod4$$
Finally, $\frac{1}{2^{v_2(S)}}(s_0+s_1A+s_2e(b)+s_3e(c))$ and then $S''$ can be decided a sum of two squares or not 
in $\mathbb{Q}_2(\sqrt{5})$ by Lemma 0.2.

(3)If $b\equiv1\pmod8$, i.e, $2O_F$ completely splits, then $0\leq v_2(S)\leq 2$. Take $S''=\frac{S}{v_2(S)}$, where
replacing $\sqrt{b}$ with $e(b)$, $\sqrt{c}$ with $e(c)$, and $\sqrt{a}$ with $A$ that $A\mathrm{gcd}(b,c)\equiv e(b)e(c)\pmod{16}$,
then in $F_{\mathcal{P}}\simeq\mathbb{Q}_2$ we have
$$S''\equiv \frac{1}{2^{v_2(S)}}(s_0+s_1A+s_2e(b)+s_3e(c))\pmod4.$$
Finally, $\frac{1}{2^{v_2(S)}}(s_0+s_1A+s_2e(b)+s_3e(c))$ and then $S''$ can be decided a sum of two squares or not 
in $\mathbb{Q}_2$ by this well-known proposition:
A unit in $O_{\mathbb{Q}_2}$ is a sum of two squares if and only if its odd part is $\equiv1\pmod4$.

Concluding all process above, and 66:1 of [1], we conclude that 

\begin{theorem}
    Let $S=s_0+s_1\sqrt{a}+s_2\sqrt{b}+s_3\sqrt{c}$ be a number in a biquadratic field 
    $F=\mathbb{Q}(\sqrt{a},\sqrt{b})$, where $a,b$ are square-free integers and $ab$ is not a square,
    $c=\frac{ab}{\mathrm{gcd}(a,b)^2}$, and $(s_0,s_1,s_2,s_3)=1$, then $S$ is a sum of two squares
    in $F$ if and only if

    (1)$S$ is totally positive if $F$ is a real field;

    (2)For every prime number $p|N_{F/\mathbb{Q}}(S)$ with $p\equiv3\pmod4$ and $(\frac{a}{p})=(\frac{b}{p})=1$,
    conditions in Lemma 2.1 are satisfied.

    (3)For every prime number $p|N_{F/\mathbb{Q}}(S)$ with $p\equiv3\pmod4$ and $p|abc$, 
    say $p|a$ and $p|b$, and $(\frac{c}{p})=1$,
    conditions in Lemma 2.2 are satisfied.

    (4)If at least one of $a,b,c$ is $\equiv1\pmod8$, say, $a\equiv1(mod8)$,

    (4-1)For $b\not\equiv1\pmod4$, there is an $S'\in O_F$ satisfying $S=2^f(2S'+1)$, $f\in\mathbb{Z}$;

    (4-2)For $b\equiv5\pmod8$, $s_0+s_1A+s_2e(b)+s_3e(c)=2^fS''$, where
    $S''\equiv1,3,\frac{\pm3\pm\sqrt{5}}{2}\pmod 4$, $f=v_2(S)$;

    (4-3)For $b\equiv1\pmod8$, $s_0+s_1A+s_2e(b)+s_3e(c)=2^fS''$, where
    $S''\equiv1\pmod 4$, $f=v_2(S)$.
\end{theorem}

Finally, for most general cases that $(s_0,s_1,s_2,s_3)=1$, we have that

\begin{theorem}
    Take assumptions in Theorem 2.3. Let $\mathcal{S}=Q^2UVS$, $0\neq Q\in\mathbb{Q}$, $U$(resp. $V$)
    is the product of distinct prime numbers that $\equiv1,2$(resp. $3$)$\pmod4$,
    then $\mathcal{S}$ is a sum of two squares in $F$ if and only if

    (1)$\mathcal{S}$ is totally positive if $F$ is a real field;

    (2)For every prime number $p|N_{F/\mathbb{Q}}(S)$ with $p\equiv3\pmod4$ and $(\frac{a}{p})=(\frac{b}{p})=1$,
    either the four values in Lemma 2.1 are all $\equiv0\pmod4$ with $p\nmid V$,
    or $\equiv2\pmod4$ with $p\mid V$.

    (3)For every prime number $p|N_{F/\mathbb{Q}}(S)$ with $p\equiv3\pmod4$ and $p|abc$, 
    say $p|a$ and $p|b$, and $(\frac{c}{p})=1$,
    conditions in Lemma 2.2 are satisfied.

    (4)If at least one of $a,b,c$ is $\equiv1\pmod8$, say, $a\equiv1(mod8)$,

    (4-1)For $b\not\equiv1\pmod4$, there is an $S'\in O_F$ satisfying $VS=2^f(2S'+1)$, $f\in\mathbb{Z}$;

    (4-2)For $b\equiv5\pmod8$, $s_0+s_1A+s_2e(b)+s_3e(c)=2^fS''$, where
    $VS''\equiv1,3,\frac{\pm3\pm\sqrt{5}}{2}\pmod 4$, $f=v_2(S)$, ;

    (4-3)For $b\equiv1\pmod8$, $s_0+s_1A+s_2e(b)+s_3e(c)=2^fS''$, where
    $VS''\equiv1\pmod 4$, $f=v_2(S)$.

    (5)For every prime number $p\equiv3\pmod4$ with $p\nmid abcN_{F/\mathbb{Q}}(S)$, $p\nmid V$.
\end{theorem}

\begin{proof}
    Directly verify that $(\frac{-1,VS}{*})$ are all 1, comparing with Theorem 2.3. 
\end{proof}

\begin{remark}
    The theorem above, together with Lemma 2.1, 2,2 and their following remarks, 
    making the process of determining a number in a biquadratic field
    can be finished on computer.
\end{remark}

\begin{example}
    Let $S=7+2\sqrt{2}+2\sqrt{3}+\sqrt{6}$ in $K=\mathbb{Q}(\sqrt{2},\sqrt{3})$. 
    $2O_K$ is totally ramified, hence $K$ has only one dyadic spot.
    One can easily verify that $S$ is totally positive, and $N_{K/\mathbb{Q}}(S)=1009$, a prime $\equiv1\pmod 4$.
    Hence $S$ is a sum of two squares at all spots, and therefore globally. Actually,
    $$S=(1-\frac{\sqrt{2}}{2})^2+(1+\frac{\sqrt{6}}{2}+\sqrt{3})^2.$$
\end{example}

\begin{example}
    Let $S=2+\sqrt{-3}+\sqrt{5}-3\sqrt{-15}$ in $K=\mathbb{Q}(\sqrt{-3},\sqrt{5})$.
    $O_K$ has two dyadic spot isomorphic to $\mathbb{Q}_2(\sqrt{5})$.
    $N_{K/\mathbb{Q}}(S)=20629=7^2\times421$, where $421\equiv1\pmod4$ and $7\equiv3\pmod4$.
    Since $(\frac{3}{7})=1$, $(\frac{5}{7})=-1$, 7-adic primes are isomorphic to $\mathbb{Q}_7(\sqrt{-1})$, so no verification is needed
    in 7-adic cases. As for dyadic cases, 
    $$S''=2\sqrt{5}-1=4(\frac{\sqrt{5}-1}{2})+1\equiv1\pmod 4.$$
    Hence $S$ is the sum of two squares at both dyadic places, therefore globally. Actually,
    $$S=(\frac{5}{2}-\sqrt{5}+\frac{\sqrt{-15}}{2}+2\sqrt{-3})^2+(3\sqrt{-3}-\frac{3}{2}+\frac{\sqrt{-15}}{2}-2\sqrt{5})^2.$$
\end{example}

\begin{example}
    Let $S=-92+\sqrt{-7}+21\sqrt{17}-\sqrt{-119}$ in $K=\mathbb{Q}(\sqrt{-7},\sqrt{17})$.
    $K$ has 4 dyadic spots, all of which isomorphic to $\mathbb{Q}_2$.
    $N_{K/\mathbb{Q}}(S)=3130541$, a prime $\equiv1\pmod4$.
    So we only need to focus on dyadic cases.
    We have $e(-7)=53$ and $e(17)=23$, so $S''=-92+53+21\times23-23\times53=-775\equiv1\pmod8$.
    Hence $S$ is a sum of two squares in $K$. Actually,
    $$S=(\frac{3}{2}+\sqrt{-7}-\frac{\sqrt{-119}}{2})^2+(\frac{1}{2}-2\sqrt{-7}+\frac{\sqrt{-119}}{2})^2.$$
\end{example}

\section*{Acknowledgements}
The author is sincerely grateful for directions by Prof Hourong Qin, and support by National Natural Science Foundation of China (11971224).

\section*{References}

[1]K. S. Williams, Integers of Biquadratic Fields, Canad. Math. Bull. Vol. 13 (4), 1970, 519-526.

[2]O. O' Meara, Introduction to Quadratic Forms, Springer-Verlag, 1973.

[3]H. Qin, The Sum Of Two Squares In A Quadratic Field, Communications in Algebra, 25:1 (1997), 177-184.

[4]W. Huang, Sum of two squares in cyclic quartic fields, preprint, 2024.
\end{document}